\documentclass[11pt]{amsart}

\usepackage{amssymb}
\usepackage{latexsym}
\usepackage{amsmath}
\usepackage{amsthm}
\usepackage{amsfonts}
\usepackage{color}
\usepackage{pdfsync}
\usepackage[usenames,dvipsnames]{pstricks}
\usepackage{epsfig}
\usepackage{pst-grad} 
\usepackage{pst-plot} 
\usepackage{accents}
\newcommand{\ubar}[1]{\underaccent{\bar}{#1}}
\usepackage{dsfont}

\newcommand{\R}{\mathbb{R}}
\newcommand{\N}{\mathbb{N}}

\newcommand{\C}{\mathcal{C}}

\newcommand{\I}{\mathcal{I}}

\newcommand{\M}{\mathcal{M}}

\newcommand{\Lambd}{\Gamma}
\newcommand{\lambd}{\gamma}

\theoremstyle{plain}
\newtheorem{defi}{Definition}[section]
\newtheorem{prop}[defi]{Proposition}
\newtheorem{teo}[defi]{Theorem}

\newtheorem{lema}[defi]{Lemma}

\newtheorem{remark}[defi]{Remark}

\theoremstyle{definition}

\theoremstyle{remark}

\numberwithin{equation}{section}

\begin{document}

\title[]{On large solutions for fractional Hamilton-Jacobi equations}

\author[]{Gonzalo D\'avila}
\address{
Gonzalo D\'avila: Departamento de Matem\'atica, Universidad T\'ecnica Federico Santa Mar\'ia \\
Casilla: v-110, Avda. Espa\~na 1680, Valpara\'iso, Chile
}
\email{gonzalo.davila@usm.cl}

\author[]{Alexander Quaas}
\address{
Alexander Quaas: Departamento de Matem\'atica, Universidad T\'ecnica Federico Santa Mar\'ia \\
Casilla: v-110, Avda. Espa\~na 1680, Valpara\'iso, Chile}
\email{alexander.quaas@usm.cl}

\author[]{Erwin Topp}
\address{
Erwin Topp: Departamento de Matem\'atica y C.C., Universidad de Santiago de Chile,
Casilla 307, Santiago, Chile.
}
\email{erwin.topp@usach.cl}



\date{\today}

\begin{abstract}
We study the existence of large solutions for nonlocal Dirichlet problems posed on a bounded, smooth domain, associated to fully nonlinear elliptic  equations of order $2s$, with $s\in  (1/2,1)$, and a coercive gradient term with subcritical power $0<p<2s$. 

Due to the nonlocal nature of the diffusion, new blow-up phenomena arise within the range $0<p<2s$, involving a continuum family of solutions and/or solutions blowing-up to $-\infty$ on the boundary. This is in striking difference with the local case studied by Lasry-Lions for the case subquadratic case $1<p<2$.
\end{abstract}

\maketitle

\section{Introduction.}

In this paper we study the boundary blow-up phenomena for solutions of fractional problems with  coercive gradient with the form
\begin{equation}\label{eqLL}
(-\Delta)^s u + |Du|^p + \lambda u  =f  \quad \mbox{in} \ \Omega,
\end{equation}
where $s \in (1/2, 1)$, $0 < p < 2s, \Omega \subset \R^N$ is a bounded, open set with smooth boundary, $f \in C(\Omega)$ and $\lambda \in \R$. Here $(-\Delta)^{s}$ denotes the fractional Laplacian of order $2s$, defined for smooth functions $u:\R^N\to\R$ as
\begin{equation}\label{fracLapl}
(-\Delta)^s u(x) = {C_{N, s} \mathrm{P.V.} \int_{\R^N} \frac{u(x) - u(z)}{|x - z|^{N + 2s}}dz},
\end{equation}
whenever the integral converges. Here, the principal value  is understood in the \textsl{Cauchy principal value sense}. The constant $C_{N, s} > 0$ is a normalizing constant so that $(-\Delta)^{s}u\to -\Delta u$ as $s\to1$ (in an adequate functional framework, see for example~\cite{Hitch}). 

The question we address concerns the existence of solutions to \eqref{eqLL} that become unbounded near the boundary. We prove the existence of multiple solutions to this problem under certain assumptions on the data. 

Our method relies on the construction of solutions by Perron's method, as it is presented by Lasry and Lions in~\cite{LL}, that is
\begin{equation}\label{eqLLloc}
-\Delta u + |Du|^p + \lambda u = f \quad \mbox{in} \ \Omega,
\end{equation}
where $\Delta$ denotes the usual Laplacian, $1 < p \leq 2$, $\lambda > 0$. In this local setting, the Dirichlet problem with blow-up boundary condition associated to~\eqref{eqLLloc} is complemented by the expression
\begin{equation}\label{blow+}
\lim_{x \in \Omega, \ x \to \partial \Omega} u(x) = +\infty.
\end{equation}

It is proven in~\cite{LL} the existence of a unique large solution for the Dirichlet problem~\eqref{eqLLloc}-\eqref{blow+} for $f \in L^\infty_{loc}(\Omega)$ with an appropriate growth on the boundary. Condition $p \leq 2$ is typically referred as the \textsl{subcritical case}.

Due to the nonlocal nature of the fractional diffusion~\eqref{fracLapl}, 
the Dirichlet problem associated to~\eqref{eqLL} requires we impose a condition on $\Omega^c$. We consider here the Dirichlet condition $u = \varphi$ in $\Omega^c$,
where $\varphi: \Omega^c \to \R$ is a given function that satisfies the integrability condition $\varphi \in L^1_w(\Omega^c)$, where for measurable set $E \subset \R^N$ we denote
$$
L^1_\omega(E) := \{ u \in L^1_{loc}(\R^N) : \int_{E} |u| \omega < +\infty\}, \quad \omega(y) := \frac{1}{(1 + |y|)^{N + 2s}}.
 $$

Summarizing, a first model for Dirichlet problem we consider here takes the form
\begin{align}\label{Dirichlet1}
\left\{\begin{array}{rll}
(-\Delta)^s u+|D u|^p + \lambda u & =f & \mbox{in} \ \Omega,\\
u & =\varphi &  \mbox{in} \ \Omega^c,\\
\lim\limits_{x\in\Omega, \ x\to\partial\Omega}u & =+\infty. &
\end{array}\right.
\end{align}

The study of large solutions has a long history starting with the work of Keller and Osserman (see \cite{Keller} and \cite{Osserman}), where they studied conditions on nonlinearities $f$ in order to find unbounded solutions to
\[
\Delta u=f(u) \quad \text{in }\Omega.
\]
For recent developments and extensions see \cite{AQGJDE}, where a thorough list of references can be found. Existence of blow-up solutions for nonlocal Dirichlet problems has been studied in~\cite{CV, CFQ, aba, chen, chen2}. Of particular interest is the work of Abatangelo~\cite{aba}, where the existence of an intriguing variety of blow-up solutions, for an ample class of reaction-diffusion problems with non-homogeneous exterior data is addressed. Roughly speaking, he constructs a fractional harmonic function that blows up near the boundary with a rate given by $d^{s-1}$ (here $d= d(x)$ denotes the distance function $\mathrm{dist}(x, \partial \Omega)$ for $x \in \Omega$). These results rely on Green functions and integral formulas for the fractional Laplacian. Using similar tools, existence of weak solutions for quasilinear equations with measure ingredients was already treated by Chen and Veron in~\cite{CV}. We mention here that all the mentioned results deal with linear diffusion.

Here we understand the condition $p < 2s$ as \textsl{strictly subcritical} in the sense that the growth of the gradient is strictly less than the order of the diffusion. Our result shows a multiplicity phenomena which is in high contrast with the second-order setting, and it is a consequence of the nonlocal nature of the problem. 
For instance, if we look on the existence result in~\cite{LL}, the authors construct a Perron's solution through blow-up barriers. These are suitable powers of the distance function $d^\beta$. By a natural scaling property of the equation, the exponent $\beta = \frac{p - 2}{p - 1}$ ensures a good approximation for the problem. A logarithmic profile is found in the critical case $p=2$.

We follow the same program here to construct solutions. For introductory purposes, we concentrate on the case $u$ satisfies the homogeneous exterior condition $u = 0$ in $\Omega^c$. Extending $d$ as zero outside $\Omega$ and for $p \in (1,2s)$, we have the corresponding exponent 
\begin{equation}\label{beta}
\beta = \frac{2s - p}{1 - p},
\end{equation}
as the one making $d^\gamma$ a good ansatz for the problem~\eqref{Dirichlet1}. Since we require $d^\gamma$ to be integrable, this introduces new critical exponents of $p$, depending on the diffusive parameter $s$. By the method used to find it, we refer to this solution as a \textsl{scale} solution.
Nevertheless, the existence of blow-up fractional harmonic functions involves the existence of a continuum of solutions which are not present in the local framework. Moreover, the nonlocal phenomena also permits the existence of blow-up solutions to $-\infty$ in certain regimes of $p < 2s$. We will come back to more specific aspects of the problem later.

Next we describe the general class of operators we will consider. For $s \in (1/2, 1)$ and constants $0 < \lambd \leq \Lambd < +\infty$, we consider the class $\mathcal K$ of measurable kernels $K : \R^N \setminus \{ 0\} \to \R$ such that $K(y) = K(-y)$ for all $y$ and satisfying the ellipticity condition
\begin{equation}\label{ellipticity}
\gamma |y|^{-(n + 2s)} \leq K(y) \leq \Gamma |y|^{-(N + 2s)}, \quad y \neq 0.
\end{equation}


For each $K \in \mathcal K$, we consider the linear operator
\begin{equation}\label{lineal}
L_K u(x) := \mathrm{PV} \int_{\R^n} [u(x + y)-u(x)] K(y) dy, 	
\end{equation}
which is well defined for measurable $u: \R^N \to \R$ satisfying adequate regularity assumptions on $x$ and weighted integrability condition at infinity; typically $u \in C^{1,1}$ in a neighborhood of $x$ and $u \in L^1_\omega(\R^N)$. 
Thus, for a  two-parameter family of kernels $\{ K_{ij}\}_{i \in I, j \in J} \subset \mathcal K$, and denoting $L_{ij} := L_{K_{ij}}$ we write 
\begin{align} \label{opp}
	 \I u(x):= \inf\limits_{i\in I} \sup\limits_{j \in J} L_{ij}u (x).
\end{align}

Associated to $\I$ of this type, we consider 
\begin{equation}\label{laoo}
\lambda_0(\I) = \inf_{x \in \Omega, i \in I, j \in J} \int_{\Omega^c} K_{ij}(x - y) dy.
\end{equation}

Notice that $0 < \lambda_0(\I) < +\infty$.

We focus on the fully homogeneous class of kernels with the form
\begin{equation}\label{Kestable}
K(z) = \frac{a(z/|z|)}{|z|^{N + 2s}},
\end{equation}
for some nonnegative, measurable function $a: S^{N - 1} \to \R$. In this setting, condition~\eqref{ellipticity} turns out to be $\gamma \leq a \leq \Gamma$.

In order to describe our existence results we require to introduce two exponents related to the gradient nonlinearity, depending on the order $s$.  We write $p_i = p_i(s)$ for $i = 1, 2$ as
\begin{equation}\label{ppp}
p_1= s + \frac{1}{2} \quad\mbox{and} \quad p_2=\frac{s+1}{2-s},
\end{equation}

A third exponent $p_0=\frac{2s}{2-s}$ also emerges in our analysis, but this plays a less important role. 
We notice that for $s \in (1/2, 1)$, we have $p_1 > 1$ and
$$
p_1(s)<p_2(s)< 2s, \quad p_2(1^-)=2, \quad p_1(1/2^+)=p_2(1/2^+)=1,
$$



\medskip

Thus, our main results is the existence of boundary blow-up solutions for the Dirichlet problem
\begin{equation}\label{eq}\tag{{\bf P}}
\left \{ \begin{array}{rl} - \I (u) + |Du|^p + \lambda u = f \quad & \mbox{in} \ \Omega, \\
u  = 0 \quad & \mbox{in} \ \Omega^c,  
\end{array} \right .
\end{equation}
as it described by the following.
\begin{teo}\label{teofbdd} Let $s \in (1/2,1)$, $0 < p < 2s$, $\Omega \subset \R^N$ be a bounded domain with $C^2$ boundary, $f\in L^\infty(\Omega) \cap C(\Omega)$.
When $1 < p < 2s$, let $\beta$ as in~\eqref{beta}.

Let $\mathcal K$ be a family of symmetric kernels satisfying \eqref{ellipticity} and~\eqref{Kestable}, $\I$ a nonlinear operator with the form~\eqref{opp}, and $\lambda > -\lambda_0(\I)$. Let $p_i$ be defined as in~\eqref{ppp}, $i=0,1,2$.

Then, we have the following existence results:

\medskip
\noindent
\textsl{1.- One parameter family of solutions (close to $s$-harmonic):} If $0 < p < p_2$, there exists $\sigma > 0$ and a family of solutions $\{ u_t \}_{t \in \R, t \neq 0} \subset C^{\sigma}(\Omega)$ to \eqref{eq}, such that for each $t$ we have
$$
d^{1 - s} u_t(x) - t = O(d^{\gamma}),  
$$
for some $\gamma > 0$ depending on $p$. In particular, if $t_1 < t_2$, then
$$
u_{t_1} < u_{t_2} \quad \mbox{in} \ \Omega.
$$

Moreover, if $p$ additionally satisfies $p < p_0$, then we can take $\gamma>0$ .

\medskip
\noindent
\textsl{2.- Positive scale solution:} If $p_1 < p < p_2$, then there exists $\sigma > 0$ and a constant $T > 0$ and a function $u \in C^{\sigma}(\Omega)$ solving~\eqref{eq} such that
$$
d(x)^{-\beta}u(x) - T = O(d(x)^\gamma),
$$ 
for some $\gamma > 0$.
\medskip

\noindent 
\textsl{3.- Negative scale solution:} For $p_2 < p < 2s$, then there exist $\sigma > 0$, $T > 0$ and a solution $u \in C^{\sigma}(\Omega)$ of \eqref{eq} such that
$$
d^{-\beta}(x)u(x) + T = O(d(x)^\gamma), 
$$
for some $\gamma > 0$.

%
\end{teo}

For the reader who is familiarized with nonlocal problems, condition $\lambda > -\lambda_0(\I)$ is a condition ensuring comparison principle, see for instance~\cite{Topp}. Since we employ an approximation procedure to construct the solutions (c.f. Proposition~\ref{perron-alternativo}), comparison principle plays an important role.

As we previously mentioned, we construct barriers that blow-up (to $+\infty$) as powers of the distance function. This allows us to construct scale solutions behaving like $d^\beta$ with $\beta$ as in~\eqref{beta}. This imposes the condition $p < p_2$ in order for the power function to be integrable. In the regime $p > p_2$, the same argument allows us to get a negative blow-up solution.



Another type of large solutions, not present in the local case, emerges here. For instance in {\it Case 1}, the family $\{ u_t \}$ is constructed ``around" a blow-up, fractional harmonic function. Our key technical result (c.f. Proposition~\ref{lema.cota1}) indicates that the function $d^{s-1}$ is close to being harmonic for $\I$. In {\it Case 3}, we construct barriers perturbing $d^{s-1}$ with other power-type functions of lower order, leading to the existence of solutions that diverge to $-\infty$.

There are plenty of open questions after this work that we believe deserve to be investigated. For instance, we could not cover the natural critical exponents as $p=p_2$ and $p = 2s$. None of the solutions found in Theorem~\ref{teofbdd} converge to a blow-up (or blow-down) solution of~\eqref{eq} when $p \to p_2$. In fact, they surprisingly converge to the unique bounded viscosity solution to~\eqref{eq}, see Remark~\ref{rmk1}. On the other hand, the critical case $p = 2s$ resembles the case $p=2$ in~\eqref{eqLLloc}, for which a logarithmic blow-up profile for the solution is obtained see \cite{LL}. In our context, the construction of barriers with a log profile is hard to handle at a technical level, and therefore we did not pursue it in this work. 

A question about applications that emerges here has to do with the connection of problem~\eqref{eq} with stochastic optimal control problems. In the local setting presented in~\cite{LL}, the solution to~\eqref{eqLLloc} turns out to be the value function of an infinite horizon stochastic optimal control problem with a cost involving $f$ and a “feedback" term depending on $p$. The admissible drifts  are those preventing the trajectories of the stochastic  process to exit the domain, leading to the blow-up of the associated value function.
Here, for each $K$ like in~\eqref{Kestable}, its associated linear operator $L_K$ is the infinitesimal generator of a $s$-stable pure jump L\'evy process, see for instance~\cite{B, S, ROS}. The connection of fractional Hamilton-Jacobi equations and stochastic optimal control problems with jumps have been shown to be hold in some cases, for instance in the unrestricted state case $(\Omega = \R^N)$. The fully nonlinear structure of~\eqref{eq} involves SDE's with controlled random parts, see for instance~\cite{Pham, S}. 


The paper is organized as follows: In Section \ref{secper} we provide the notion of solution we use here and a general Perron's method. In Section \ref{sectech} we provide precise estimates of the nonlocal operator applied to powers of the distance function. In Section \ref{secteo} we provide the proof of Theorem \eqref{teofbdd}, which is accomplished by constructing sub and super solutions based on the results of Section \ref{sectech}. Finally in Section \ref{secext} we provide some extensions that include unbounded right-hand side and non-homogeneous exterior Dirichlet data.

\section{Perron's Solutions}\label{secper}

Given $\varphi \in L^1_\omega(\Omega^c)$, $\lambda > -\lambda_0(\I)$ and $f \in C(\Omega)$, we are interested in viscosity solutions to the problem
\begin{align}\label{eqgeneral}
\left\{\begin{array}{rll}
-\mathcal I u+|D u|^p + \lambda u & =f & \mbox{in} \ \Omega,\\
u & =\varphi &  \mbox{in} \ \Omega^c,\\
\lim\limits_{x\in\Omega, \ x\to\partial\Omega}u & =+\infty, &
\end{array}\right.
\end{align}
and its blow-up version to minus infinity, that is, $\lim\limits_{x\in\Omega, \ x\to\partial\Omega}u = -\infty$.

We start with the notion of viscosity solution, see~\cite{BChI}.
\begin{defi}\label{defsol}
A function $u \in L^1_\omega(\R^N)$, upper semicontinuous in $\Omega$, is a viscosity subsolution to the Dirichlet problem
\begin{align}\label{eqdef}
\left\{\begin{array}{rll}
-\mathcal I u+|D u|^p + \lambda u & =f & \mbox{in} \ \Omega,\\
u & =\varphi &  \mbox{in} \ \Omega^c,
\end{array}\right.
\end{align}
 if $u \leq \varphi$ in $\Omega^c$, and, for every $x_0 \in \Omega$ and every function $\phi \in L^1_\omega(\R^N) \cap C^2(\Omega)$ such that $u(x_0) = \phi(x_0)$, $u \leq \phi$ in $B_\delta(x_0)$ for some $\delta > 0$, we have the inequality
\begin{equation*}
-\I u^\phi_{\delta, x_0}(x_0) + |D \phi(x_0)|^p + \lambda u(x_0) \leq f(x_0),
\end{equation*}
where $u^\phi_{\delta, x_0}: \R^N \to \R$ is the function defined as $u^\phi_{\delta, x_0}(x) = \phi(x)$ in $B_\delta(x_0)$, $u^\phi_{\delta, x_0}(x) = u(x)$ in $B_\delta^c(x_0)$.

A function $u \in L^1_\omega(\R^N)$, lower semicontinuous in $\Omega$, is a viscosity supersolution to the Dirichlet problem~\eqref{eqdef}
 if $u \geq \varphi$ in $\Omega^c$ and for every $x_0 \in \Omega$ and every function $\phi \in L^1_\omega(\R^N) \cap C^2(\Omega)$ such that $u(x_0) = \phi(x_0)$, $u \geq \phi$ in $B_\delta(x_0)$ for some $\delta > 0$, we have the inequality
\begin{equation*}
-\I u^\phi_{\delta, x_0}(x_0) + |D \phi(x_0)|^p + \lambda v(x_0) \geq f(x_0),
\end{equation*}
where $u^\phi_{\delta, x_0}$ is defined as before.

A function $u \in L^1_\omega(\R^N) \cap C(\Omega)$ is a solution to~\eqref{eqdef} if $u = \varphi$ in $\Omega^c$ and is simultaneously a viscosity sub and supersolution to the problem.

Finally, we say that $u$ is an strict subsolution (resp. supersolution) to~\eqref{eqdef} if there exists $\epsilon > 0$ such that $u$ satisfies the viscosity inequality with $f(x_0)-\epsilon$ (resp. $f(x_0) + \epsilon$) instead of $f(x_0)$, for all $x_0 \in \Omega$. 
\end{defi}

Existence and uniqueness for solutions $u \in C(\R^N)$ can be found in~\cite{BChI}, in particular it attains the boundary data imposed by $\varphi$.
However, Definition~\ref{defsol} allows the possibility to have solutions  which are unbounded in $\Omega$. We use an approximation procedure based on Perron's method.
\begin{prop}\label{perron-alternativo}
Let $\Omega$ be a bounded, smooth domain, $f \in C(\Omega)$, $0<p\leq2s$ and $\lambda>-\lambda_0(\I)$. Suppose there exist a supersolution $\bar U$ and a subsolution $\ubar U$ of \eqref{eqgeneral} with $\bar U, \ubar U\in C(\Omega)\cap L^1_\omega(\R^N)$ with $\underline U = \bar U = \varphi$ in $\Omega^c$, and such that one of them is strict. Furthermore assume that 
\begin{align}\label{ex}
\bar U\geq\ubar U \quad \mbox{in} \ \R^N, \quad \lim\limits_{x\in\Omega,\ x\to\partial\Omega}\ubar U=+\infty.
\end{align}

Then there exists a solution $u \in C^\alpha(\Omega) \cap L^1_w(\R^N)$ of \eqref{eqgeneral} satisfying $\ubar U\leq u\leq \bar U$. 

An analogous result can be stated for sub and supersolutions $\underline U \leq \bar U$ with $\bar U(x) \to -\infty$ as $x \to \partial \Omega$.
\end{prop}

\begin{proof}
We assume $\underline U$ is a strict subsolution, the other case follows the same lines.
Let $\Omega_n = \{x\in\Omega : \text{dist}(x,\partial\Omega)>1/n\}$. For $k, n \in \N$, let $w_{n, k}$ be a continuous function in $\bar \Omega \setminus \Omega_{n + k + 1}$ such that $w_{n,k} = \varphi$ on $\partial \Omega$, $w_{n,k} = \underline U$ on $\partial \Omega_{n + k + 1}$ (say, the harmonic function in $\Omega \setminus \bar \Omega_{n + k+ 1}$ satisfying the mentioned boundary conditions).

Now, let $U_{n,k} : \Omega_{n}^c \to \R$ given by
\begin{equation*}
U_{n,k}(x) = \left \{ \begin{array}{ll} \underline U(x) \quad & \mbox{if} \ x \in \Omega_{n + k + 1} \setminus \Omega_n, \\
\min\{ w_{n,k}(x), \underline U(x) \} \quad & \mbox{if} \ x \in \Omega \setminus \Omega_{n + k + 1}, \\ 
\varphi(x) & \mbox{if} \ x \in \Omega^c. \end{array} \right .
\end{equation*}

Notice that $U_{n,k}$ is continuous, and $\min_{\partial \Omega} \{ \varphi \} \leq U_{n, k} \leq \underline U$ in $\Omega \setminus \Omega_{n + k + 1}$, from which, by Dominated Convergence Theorem, we have that
\begin{equation*}
\int_{\Omega \setminus \Omega_{n + k + 1}} |U_{n,k}(y) - \underline U(y) |K(x - y) dy \to 0,
\end{equation*}
as $k \to \infty$, uniformly in $x \in \Omega_n, K \in \mathcal K$, for $n$ fixed. Then, since $\underline U$ is a strict subsolution, the above estimate implies that for each $n$, there exists $k(n)$ such that, for each $k \geq k(n)$, the function $U_{n,k}$ is a viscosity subsolution to the Dirichlet problem
\begin{align}\label{eqkn}
\left \{ \begin{array}{rll} -\mathcal I u+|Du|^p+\lambda u&=f - \frac{1}{n} \quad & \text{in }\Omega_n, \\
u&= U_{n,k} \quad & \text{in }\Omega_n^c, \end{array} \right .
\end{align}
By a similar argument, using that $\bar U$ is a viscosity supersolution for the problem in $\Omega$, we can construct $\bar U_{n,k}  \in C(\R^N)$ a supersolution to~\eqref{eqkn} with $\bar U_{n,k} = \bar U$ in $\Omega_{n + k + 1}$, $U_{n,k} \leq \bar U_{n,k}$ in $\Omega$, and such that $\bar U_{n,k} = \varphi$ in $\Omega^c$, for all $k \geq k(n)$ (relabeling $k(n)$ if necessary). Thus, by Theorem 1 in~\cite{BChI}, there exists a unique viscosity solution $u_{n,k} \in C(\R^N)$ for~\eqref{eqkn}. Moreover it satisfies $\underline U \leq u_{n,k} \leq \bar U$ in $\Omega_{n + k + 1}$, for all $n$ and $k \geq k(n)$, and by construction we have $u_{n,k} \in L^1_\omega(\R^N)$ uniformly in $n$ and $k \geq k(n)$. Comparison principles are available by the assumption $\lambda > -\lambda_0 (\I)$ for all $n, k$ large.

%
%
%
%
%

Thus, the family $\{ u_{n,k(n)} \}_{n \in \N}$ have uniform interior $C^{\alpha}$ estimates by the results of~\cite{bci11}.
Using this and that the family is uniformly bounded in compact sets of $\Omega$, we can use stability results of viscosity solutions to conclude the result, taking $n \to +\infty$.
\end{proof}

\begin{remark}
We notice that the above result holds if we assume that $\bar U$ (resp. $\bar U$) is a viscosity subsolution (resp. supersolution) to~\eqref{eqgeneral} which is strict in each compact subset of $\Omega$.
\end{remark}

\section{Technical lemmas}\label{sectech}
\label{secciontecnica}

We use the notation $d: \R^N \to \R$ such that $d(x) = \mathrm{dist}(x, \partial \Omega)$ for $x \in \Omega$, and $d(x) = 0$ for $x \in \Omega^c$. 
Since the domain is smooth, we have the existence $\delta_0 > 0$ such that $d$ is a $C^2$ function on the set $\Omega_{\delta} = \{ x \in \Omega : d(x) < \delta \}$ for all $\delta < \delta_0$. 
Given $\tau \in (-1, 2s)$, we denote $d^\tau: \R^N \to \R$ such that $d^\tau(x) = (d(x))^\tau$ for $x \in \Omega$ and zero in $\Omega^c$, with the convention and $d^0 = \chi_\Omega$.

For a function $u: \R^N \to \R$ measurable, $A \subset \R^N$ measurable set and $x \in \R^N$, we denote
\begin{equation*}
L_K[A]u(x) = \mathrm{P.V.} \int_A[u(x + z) - u(x)]K(z)dz.
\end{equation*}

Here we only assume~\eqref{ellipticity}. For $K$ in this class and $\rho > 0$, we introduce the notation
\begin{equation}\label{Krho}
K^\rho(z) = \rho^{N + 2s} K(\rho z), \quad z \neq 0.
\end{equation}

Notice that $K^\rho$ satisfies~\eqref{ellipticity} with the same ellipticity constants as $K$, and that if $K$ satisfies~\eqref{Kestable}, then $K^\rho = K$.


The main result of this section is the following
\begin{prop}\label{lema.cota1}
	Let $\Omega \subset \R^N$ be a bounded domain with $C^2$ boundary, $s\in (0,1)$, and let $\mathcal K$ a family of kernels satisfying~\eqref{ellipticity}. 
	
	Then, for each $\tau \in (-1, 2s)$, there exist $\delta > 0$ such that
	$$
	\I d^{\tau}(x) = d^{\tau - 2s}(x) (c(d(x), \tau) + O(d(x)^s)), \quad x \in \Omega_\delta,
	$$
	where
$$
c(d(x), \tau) 
= \inf_{i \in I} \sup_{j \in J} \mathrm{P.V.}  \int_{\R^N} [(1 + z_N)_+^\tau - 1]K_{ij}^{d(x)}(z)dz.
$$

Finally, if we additionally assume~\eqref{Kestable}, then $c( \I, d(x), \tau)= c(\tau)$ and this constant satisfies $c(-1^+) =+\infty$, $c(2s^-) = +\infty$, $c(s - 1) = c(s) = 0$, $c(\tau) > 0$ if $\tau \in (-1, s-1) \cup (s, 2s)$ and $c(\tau) < 0$ for $\tau \in (s-1, s)$.
\end{prop}

Before we continue with the proof of the proposition we will introduce some notation. When~\eqref{Kestable} holds, then we denote
$$c(\I, \tau)=c( \I, d(x), \tau) $$ 
and we will omit the dependence on $\I$ whenever the context is clear.
In particular,  if we define $\tilde \I$ as the operator $-  I (- \cdot)$ then it satisfies  ~\eqref{Kestable} so we will denote  $\tilde c(\tau)=c( \tilde \I, d(x), \tau) $. 
Moreover, since $\M^\pm$ also satisfies ~\eqref{Kestable}  we will denote $c^\pm(\tau)=c( \M^\pm, d(x), \tau) $.

\medskip

We recall that for each $x \in \partial \Omega$, there exists an open set $\mathcal U \subset \R^{N - 1}$ containing the origin, $r > 0$ and  a $C^2$ function $\psi_x: \mathcal U \to \R$ such that $\partial \Omega \cap B_{\R^N}(x,r) \subset \mathcal R_x \{ x + (z', \psi_z(z')) : z' \in \mathcal U \}$, for some rotation matrix $\mathcal R_x$. By compactness and regularity of $\partial \Omega$, we have a finite number of charts covering $\partial \Omega$ with uniform $C^2$ bounds.

Now, let $x \in \Omega$, and denote $\rho = d(x)$. After rotation, we assume $x = \rho e_N$, the projection of $x$ to $\partial \Omega$ is the origin, and therefore, that the local chart $\psi$ corresponding to this point satisfies $\psi(0) = 0, D\psi(0) = 0$. Thus, we have the existence of $C_\Omega > 0$ such that
\begin{equation}\label{psi}
|\psi(z')| \leq C_\Omega |z'|^2.
\end{equation}

The key technical step to prove Proposition~\ref{lema.cota1} is the following
\begin{lema}\label{lematecnico}
Let $K \in \mathcal K$ satisfying~\eqref{ellipticity}, and $\tau \in (-1, 2s)$. Let $\eta \in (0,1)$, $Q_\eta = B_\eta' \times (-\eta, \eta) \subset \R^N$. Let $x = (0', \rho) \in \Omega$ such that the projection of $x$ to $\partial \Omega$ is the origin, and that $x \in Q_{\eta}$. Denote
	$$
	I := \lim_{\epsilon \to 0} \int_{Q_\eta \setminus B_\epsilon} [(\rho + y_N - \psi(y'))_+^\tau - \rho^\tau] K(y) dy.
	$$

Then, the limit exists. For each $\eta > 0$ small enough, and all $\rho > 0$ small enough in terms of $\eta$, we have the expansion 
 \begin{equation}\label{claim1}
	I = \rho^{\tau - 2s} (c_K(\rho, \tau) + O(\rho^s) + O(\rho^{\tau + 1})),	
	\end{equation}
	where
	$$
	c_K(\rho, \tau) = \mathrm{P.V.} \int_{\R^N} [(1 + z_N)_+^\tau - 1] K^\rho(z) dz,
	$$
and the $O$ terms depend only on $N,s, \Omega, 1 + \tau$, $\eta$ and the ellipticity constants.
\end{lema}

We use this estimate to prove our main result of this section.
\begin{proof}[Proof of Proposition~\ref{lema.cota1}]
For an arbitrary linear operator $L = L_K$ in the family, and for $\eta$ as in Lemma~\ref{lematecnico}, we write 
\begin{equation}\label{primera}
L d^\tau (x) = L[Q_\eta] d^\tau(x) + L[Q_\eta^c] d^\tau(x).
\end{equation}

It is easy to see that if $\rho < \eta/4$, we have
\begin{equation*}
|L[Q_\eta^c] d^\tau(x)| \leq \int_{Q_\eta^c} |d^\tau(x + y) - \rho^\tau|K(y)dy \leq C\Lambda (c_\tau \eta^{-(N + 2s)} + \rho^\tau \eta^{-2s}),
\end{equation*}
for some $C > 0$ just depending on $N, s, \Omega$. 

From now on we concentrate on $L[Q_\eta] d^\tau(x)$ in~\eqref{primera}. For each $z = (z', z_N) \in \Omega \cap Q_\eta$, we have
$$
d(z) \leq z_N - \psi(z'), 
$$
from which we directly have
\begin{equation}\label{Ibelow}
L d^\tau (x) \geq I,
\end{equation}
with $I$ as in Lemma~\ref{lematecnico}.

On the other hand, by the smoothness of the domain, we use Lemma 3.1 in~\cite{CFQ}, from which we get the existence of a constant $C_\Omega > 0$ just depending on $\Omega$ and $N$ such that
\begin{equation*}
d(y) \geq (y_N - \psi(y'))(1 - C_\Omega|y'|^2),
\end{equation*}
for $y$ close to the boundary, near $x_0$. Thus, taking $\eta$ small enough in terms of $C_\Omega$, we also have
\begin{equation*}
L d^\tau (x) \leq \lim_{\epsilon \to 0^+} \int_{Q_\eta \setminus B_\epsilon} [(\rho + y_N - \psi(y'))_+^\tau (1 + C_\Omega |y'|^2) - \rho^\tau] K(y) dy, 
\end{equation*}
and from here, it is easy to see that
\begin{equation}\label{Ldtau0}
\begin{split}
L d^\tau (x) \leq I + C_\Omega \int_{Q_\eta} (\rho + y_N - \psi(y'))^\tau_+ |y'|^2 K(y)dy.
\end{split}
\end{equation}

Now, for the second term in the last expression, we can write
\begin{align*}
& \int_{Q_\eta} (\rho + y_N - \psi(y'))^\tau_+ |y'|^2 K(y)dy \\
\leq & \Lambda \int_{B_\eta'} |y'|^{2 - N - 2s} \int_{\psi(y') - \rho}^{\eta} (\rho + y_N - \psi(y'))^\tau dy_N dy' \\
\leq & \Lambda \frac{1}{1 + \tau} (\rho + \eta/4)^{1 + \tau}\int_{B_\eta'} |y'|^{2 - N - 2s} dy' \\
\leq & C \frac{\eta^{1 + \tau}}{1 + \tau},
\end{align*}
where we have used the fact that $\psi(z') \leq \eta/4$ for $|z'| \leq \eta$. Hence, replacing in~\eqref{Ldtau0} and using~\eqref{Ibelow}, we conclude that
\begin{align}\label{Iabove}
I \leq L d^\tau (x) \leq & I + C C_\Omega \frac{1}{1 + \tau},
\end{align}
where the constant $C > 0$ depends on $N, s, \Lambda$

Then, by Lemma~\ref{lematecnico}, we get	\begin{equation*}
	L d^\tau(x) = \rho^{\tau - 2s} (c_K(\rho, \tau) + O(\rho^s) + O(\rho^{1 + \tau}) + O(\rho^{2s - \tau})).
	\end{equation*}
	
We get from here that
	\begin{equation*}
	\I d^\tau(x) = \rho^{\tau - 2s} \Big{(} \inf_{i \in I} \sup_{j \in J}  c_{K_{ij}}(\rho, \tau)  + O(\rho^s) + O(\rho^{1 + \tau}) \Big{)},
	\end{equation*}
	from which the first result follows.
	
	For the last part of the proposition, since the operator  are of the form~\eqref{Kestable} we have $K_{ij}^\rho = K_{ij}$, then  Lemma 2.1 in~\cite{ROS} we arrive at 
        \begin{equation*}
        c(\tau):=\inf_{i \in I} \sup_{j \in J}  c_{K_{ij}}(\tau) = \inf_i \sup_j \{ -\tilde a_{ij} \ (-\Delta_{\R})^s w_\tau (1) \}=:a^*(-\Delta_{\R})^s w_\tau (1) 
        \end{equation*}
        where $(\Delta_{\R})^s$ denotes the fractional Laplacian in dimension one, $w_\gamma(t) = t_+^{\tau}$ for $t \in \R$, and 
        $$
        \tilde a_{ij} := \int_{S^{N - 1}} |\theta_N|^{2s} a_{ij}(\theta) d\sigma(\theta),
        $$
        where $\sigma$ denotes the $N-1$ dimensional Hausdorff measure in the unit sphere. 
Now if we define   $c_1(\tau)= (-\Delta_{\R})^s w_\tau (1) $ the qualitative properties follows since the function $c_1$ strictly concave in (-1,2s) by Proposition 3.1 of \cite{CFQ} and $c_1(-1^+) =+\infty$, $c_1(2s^-) = +\infty$. Moreover,  $c_1(s)= c_1(s-1) = 0$ by Lemma 6.2 of ~\cite{ROS}.

\end{proof}

The rest of the section is devoted to the
\begin{proof}[Proof of Lemma~\ref{lematecnico}] We concentrate on the case $\tau < 0$ since it is the most difficult due to the unboundedness of $d^\tau$.

We write
$$
I = I_0 + \tilde I,
$$
with 
\begin{align*}
I_0 = & \mathrm{P.V.} \int_{Q_\eta} [(\rho + y_N)_+^\tau - \rho^{\tau}] K(y)dy, \\
\tilde I = & \int_{Q_\eta} [(\rho + y_N - \psi(y'))_+^\tau - (\rho + y_N)_+^\tau] K(y)dy,
\end{align*}
where the last integral is well-defined since, using~\eqref{psi}, we have
$$
|(\rho + y_N - \psi(y'))_+^\tau - (\rho + y_N)^\tau| \leq C\rho^{\tau - 2} |y|^2, 
$$
for all $|y|$ small enough in terms of $\rho$. This is enough to compensate the singularity of the kernel $K$ and pass to the limit as $\epsilon \to 0$ using Dominated Convergence Theorem.

\medskip

\noindent
\textsl{1.- Estimate for $\tilde I$.} Rescaling, we have
	$$
	\tilde I =  \rho^{\tau - 2s} \int_{Q_{\eta/\rho}} [(1 + z_N - \tilde \psi(z'))_+^\tau - (1 + z_N)_+^\tau] K^\rho(z)dz,
	$$
	where $\tilde \psi(z') := \rho^{-1} \psi(\rho z')$ and $K^\rho(z) = \rho^{N + 2s} K(\rho z)$. Notice that by the ellipticity condition we have
	\begin{equation*}
	\gamma |z|^{-(N + 2s)} \leq K^\rho(z) \leq \Lambda |z|^{-(N + 2s)}, \quad z \in \R^N \setminus \{ 0 \}, \ \rho > 0.
	\end{equation*}
	
	For this, we divide the integral in several parts. Namely, we consider the splitting
	\begin{align}\label{splittildeI}
	d^{2s - \tau} \tilde I = I_1 + I_2 + I_3,
	\end{align}
	where for $i = 1,2,3$ we denote
	\begin{equation*}
	\begin{split}
	& I_i = \int_{A_i} [(1 + z_N - \tilde \psi(z'))_+^\tau - (1 + z_N)_+^\tau] K^\rho(z)dz, \ \mbox{with} \\
	&  A_1 = B_1' \times (-\frac{\eta}{\rho}, \frac{\eta}{\rho}), \ A_2 = B_{\frac{\eta}{\sqrt{\rho}}}' \setminus B_1' \times (-\frac{\eta}{\rho}, \frac{\eta}{\rho}), \ A_3 = Q_{\frac{\eta}{\rho}} \setminus (A_1 \cup A_2).
	\end{split}
	\end{equation*}

For $I_1$, we notice that by the assumptions on the chart $\psi$ (c.f.~\eqref{psi}), we have $|\tilde \psi(z')| \leq \rho |z'|^2$ if $|z'| \leq 1$. Then, we perform the subdivision 
\begin{align*}
& I_1 = I_{11} + I_{12} + I_{13}, \\
& A_{11} = B_1' \times (-\frac{\eta}{\rho}, -\frac{1}{2}), \ A_{12} = B_1' \times  (-\frac{1}{2}, \frac{1}{2}), \ A_{13} = B_1' \times  (\frac{1}{2}, \frac{\eta}{\rho}),
\end{align*}
where we have adopted a similar notation as in~\eqref{splittildeI}. 

For $I_{11}$, we make a subdivision with the form
\begin{align*}
I_{11} \leq &  \int_{B_1'} \int_{\tilde \psi_-(z') - 1}^{-1} (1 + z_N - \tilde \psi(z'))^\tau K^\rho(z)dz \\
&+ \int_{B_1'} \int_{\tilde \psi_+(z') - 1}^{-1} [(1 + z_N - \tilde \psi_+(z'))^\tau - (1 + z_N)^\tau]K^\rho(z)dz \\
=: & I_{111} + I_{112},
\end{align*}
where, for $a \in \R$, we have written $a = a_+ + a_-$.

Using the ellipticity condition and integrating by parts, we can write
\begin{align*}
I_{112} \leq&  \Gamma \int_{B_1'} \int_{\tilde \psi_+(z') - 1}^{-1/2} \frac{(1 + z_N - \tilde \psi_+(z'))^\tau - (1 + z_N)^\tau}{|z|^{N+2s}} dz_N dz' \\
= &  \frac{\Gamma}{1 + \tau} \int_{B_1'} \Big{\{} \frac{(1 + z_N - \tilde \psi_+(z'))^{\tau + 1} - (1 + z_N)^{\tau + 1}}{|z|^{N+2s}} \Big{|}_{z_N = \tilde \psi_+(z') - 1}^{z_N = -1/2} \\
& \qquad \qquad + (N + 2s) \int_{\tilde \psi_+ - 1}^{-1/2} \frac{(1 + z_N - \tilde \psi_+(z'))^{\tau + 1} - (1 + z_N)^{\tau + 1}}{|z|^{N+2s +1}} z_N dz_N \Big{\}}dz' \\ 
\leq &  \frac{C \Gamma}{1 + \tau} \int_{B_1'} \Big{\{} \frac{(\tau + 1) \rho |z'|^2 + \rho^{\tau + 1} |z'|^{2(\tau + 1)}}{(1 + |z'|^2)^{(N+2s)/2}} \\
& \qquad \qquad + \int_{\tilde \psi_+(z') - 1}^{-1/2} \frac{(1 + z_N - \tilde \psi_+(z'))^{\tau + 1} - (1 + z_N)^{\tau + 1}}{|z|^{N+2s +1}} z_N dz_N \Big{\}}dz' \\ 
\leq &  \frac{C \Gamma}{1 + \tau}  \Big{\{} \rho^{\tau + 1} +
\int_{B_1'} \int_{\tilde \psi_+(z') - 1}^{-1/2} \frac{(1 + z_N - \tilde \psi_+(z'))^{\tau + 1} - (1 + z_N)^{\tau + 1}}{|z|^{N+2s +1}} z_N dz_N dz' \Big{\}},
\end{align*}
for some universal constant $C > 0$. Here we have used that $|\tilde \psi(z')| \leq C \rho |z'|^2$ for $|z'| \leq 1$. Since $1 + \tau > 0$, we see that
\begin{align*}
I_{112} \leq &  \frac{C \Gamma}{1 + \tau}  \Big{\{} \rho^{\tau + 1} +
\int_{B_1'} \frac{\rho |z'|^2}{(1 + |z'|^2)^{\frac{N + 2s + 1}{2}}} \int_{\tilde \psi_+ - 1}^{-1/2} (1 + z_N)^\tau dz_N dz' \Big{\}},
\end{align*}
and from this we conclude that
\begin{equation*}\label{I11upper}
I_{112} \leq \frac{C \Gamma}{1 + \tau} \rho^{\tau + 1}.
\end{equation*}

For $I_{111}$, by direct integration and the estimates for $\tilde \psi$ we see that
\begin{align*}
I_{111} \leq C\Gamma \int_{B_1'} |\tilde \psi(z')|^{1 + \tau} dz \leq C \Gamma \rho^{1 + \tau}.
\end{align*}

For the lower bound we proceed similarly, noticing that this time we can write
\begin{equation}\label{cotainf}
\begin{split}
I_{11} \geq & -\int_{B_1'} \int_{-1}^{- 1 + \tilde \psi_+(z')} (1 + z_N)^\tau K^\rho(z)dz \\
& + \int_{B_1'} \int_{- 1 + \tilde \psi_+(z')}^{-1/2} [(1 + z_N - \tilde \psi_+(z'))^\tau - (1 + z_N)^\tau] K^\rho(z)dz 
\end{split}
\end{equation}
From here, by direct integration in the first term in the last inequality, and a similar procedure leading to the estimate concerning $I_{112}$ above for the second term, we conclude that
\begin{equation}\label{I11}
-\frac{C\Gamma}{1 + \tau} \rho^{\tau + 1} \leq I_{11} \leq \frac{C \Gamma}{1 + \tau} \rho^{\tau + 1},
\end{equation}

\medskip

For $I_{12}$, we use that $|\tilde \psi(z')| \leq C \rho |z'|^2$ to perform a first-order Taylor expansion to get
\begin{align}\label{I12}
-C \Gamma \rho \leq I_{12} \leq C \Gamma \int_{B_1'} \int_{-1/2}^{1/2} \frac{\rho |z'|^2}{|z|^{N + 2s}} dz_N dz' \leq C \Gamma \rho.
\end{align}

\medskip

For $I_{13}$, we perform a Taylor expansion again, from which we can write
\begin{align*}
I_{13} & \leq C\Gamma \rho \int_{B_1'} |z'|^2 \int_{1/2}^{\eta/\rho} \frac{dz_N}{(|z_N|^2 + |z'|^2)^{\frac{N + 2s} {2}}}dz' \\
& \leq C\Gamma \rho \int_{B_1'} |z'|^2 |z'|^{-(N + 2s) + 1} \int_{0}^{+\infty} \frac{dt}{(t^2 + 1)^{\frac{N + 2s} {2}}}dz'.
\end{align*}
A similar lower bound can be easily obtained, from which we conclude that
$-C \Gamma \rho \leq I_{13} \leq C\Gamma \rho.$ Gathering this estimate together with~\eqref{I12} and~\eqref{I11} lead us to 
\begin{align}\label{I1}
I_1 = O(\rho^{\tau + 1})
\end{align}

\medskip

Now we proceed with $I_{2}$ in~\eqref{splittildeI}. We write
\begin{align*}
& I_2 = I_{21} + I_{22}, \ \mbox{with} \\
& A_{21} = B_{\eta/\sqrt{\rho}}' \setminus B_1' \times (-\frac{\eta}{\rho}, -\frac{1}{2}), \ A_{22} = B_{\eta/\sqrt{\rho}}' \setminus B_1' \times (-\frac{1}{2}, \frac{\eta}{\rho}),
\end{align*}
where we have adopted a similar notation as in~\eqref{splittildeI}. 

We start by noticing that $|\tilde \psi(z')| \leq C \eta^2$ when $|z'| \leq \eta/\sqrt{\rho}$, and from here, fixing $\eta > 0$ universally small, we can write
\begin{align*}
I_{21} \leq &  \int_{B_{\eta/\sqrt{\rho}}' \setminus B_1'} \int_{\tilde \psi_-(z') - 1}^{-1} (1 + z_N - \tilde \psi(z'))^\tau K^\rho(z)dz \\
&+ \int_{B_{\eta/\sqrt{\rho}}' \setminus B_1'} \int_{\tilde \psi_+(z') - 1}^{-1/2} [(1 + z_N - \tilde \psi_+(z'))^\tau - (1 + z_N)^\tau]K^\rho(z)dz \\
=: & I_{211} + I_{212},
\end{align*}

For $I_{211}$ we have
\begin{align*}
I_{211} \leq & C \Gamma \int_{B_{\eta/\sqrt{\rho}}' \setminus B_1'}  |z'|^{-(N + 2s)} \int_{\tilde \psi_-(z') - 1}^{-1} (1 + z_N - \tilde \psi(z'))^\tau dz_N dz' \\
\leq & C \Gamma \rho^{\tau + 1} \int_{B_{\eta/\sqrt{\rho}}' \setminus B_1'}  |z'|^{-(N + 2s) + 2(\tau + 1)} dz' \\
\leq & \frac{C}{1 - 2s}\Gamma \rho^{1 + \tau},
\end{align*}

For $I_{212}$, by taking $\eta$ small enough but independent of $\rho$, we have
\begin{align*}
I_{21} & \leq C \Lambda \int_{B_{\eta/\sqrt{\rho}}' \setminus B_1'} \frac{1}{(1 + |z'|)^{N+2s}} \int_{\tilde \psi_+(z') - 1}^{-1/2} [(1 + z_N - \tilde \psi_+(z'))^\tau - (1 + z_N)^\tau] dz_N dz' \\
& \leq C \frac{\Lambda}{1 + \tau} \int_{B_{\eta/\sqrt{\rho}}' \setminus B_1'} \frac{1}{(1 + |z'|)^{N+2s}} \Big{(}(\frac{1}{2} - \tilde \psi_+(z'))^{\tau + 1} - (\frac{1}{2})^{\tau + 1} + \tilde \psi_+(z')^{\tau + 1}\Big{)} dz' \\
& \leq C \frac{\Lambda}{1 + \tau} \Big{(} \rho \int_{B_{\eta/\sqrt{\rho}}' \setminus B_1'} \frac{|z'|^2}{(1 + |z'|)^{N+2s}} dz' +  \rho^{\tau + 1} \int_{B_{\eta/\sqrt{\rho}}' \setminus B_1'} \frac{|z'|^{2(\tau + 1)}}{(1 + |z'|)^{N+2s}} dz' \Big{)} \\
& \leq C \rho^{1 + \tau},
\end{align*}
and from here, collecting the above estimates, we conclude that $I_{21} \leq C \Gamma \rho^{1 + \tau}$. In a similar fashion as in~\eqref{cotainf} but applied to $I_{21}$, we arrive at the estimate
\begin{equation}\label{I21}
I_{21} = O(\rho^{1 + \tau}).
\end{equation}

For $I_{22}$, taking $\eta$ universally small, there exist $0 < c, C < +\infty$ such that
$$
c (1 + z_N)^{\tau - 1} |\tilde \psi_+(z')| \leq (1 + z_N - \tilde \psi_+(z'))^\tau - (1 + z_N)^\tau = C (1 + z_N)^{\tau - 1} |\tilde \psi_+(z')|,
$$
from which we can write
\begin{align*}
I_{22} = O(1) \int_{B_{\eta/\sqrt{\rho}}' \setminus B_1'} \frac{\rho |z'|^2}{(1 + |z'|)^{N+2s}} \int_{-1/2}^{+\infty} (1 + z_N)^{\tau - 1} dz_N dz',
\end{align*}
and from here we conclude that $I_{22} = O(\rho)$. This together with~\eqref{I21} lead us to
\begin{equation}\label{I2}
I_{2} = O(\rho^{1 + \tau}). 
\end{equation}

Now we deal with $I_3$. This time we consider the splitting
\begin{align*}
& I_3 = I_{31} + I_{32}, \ \mbox{with} \\
& A_{31} = \{ (z', z_N) : z' \in B_{\eta/\rho}' \setminus B_{\eta/\sqrt{\rho}}', \ z_N \in (-\eta/\rho, \tilde \psi_+(z') + 1) \}, \\
& A_{32} = \{ (z', z_N) : z' \in B_{\eta/\rho}' \setminus B_{\eta/\sqrt{\rho}}', \ z_N \in (\tilde \psi_+(z') + 1, \eta/\rho) \},
\end{align*}
where we have adopted the notation in~\eqref{splittildeI}.

For $I_{31}$ we see that
\begin{align*}
I_{31} & \leq \Lambda \int_{B_{\eta/\rho}' \setminus B_{\eta/\sqrt{\rho}}'} \int_{\tilde \psi_+(z') - 1}^{\tilde \psi_+(z') + 1} \frac{(1 + z_N - \tilde \psi_+(z'))^\tau}{|z'|^{N + 2s}}dz_N dz' \\
& \leq C\Lambda \int_{B_{\eta/\rho}' \setminus B_{\eta/\sqrt{\rho}}'} \frac{1}{|z'|^{N + 2s}} \int_{0}^{2} (1 + t)^\tau dt dz' \\
& \leq C\Lambda \int_{\eta/\sqrt{\rho}}^{\eta/\rho} r^{-(2 + 2s)}dr,
\end{align*}
from which we conclude that
$I_{31} \leq C\Lambda \rho^{s + 1/2}$.

On the other hand, since $\tau < 0$, for $I_{32}$ we can write
\begin{align*}
I_{32} & \leq \Lambda \int_{B_{\eta/\rho}' \setminus B_{\eta/\sqrt{\rho}}'} \int_{\tilde \psi_+(z') + 1}^{+\infty} \frac{(1 + z_N - \tilde \psi_+(z'))^\tau}{(z_N^2 + |z'|^2)^{(N + 2s)/2}}dz_N dz' \\
& \leq \Lambda \int_{B_{\eta/\rho}' \setminus B_{\eta/\sqrt{\rho}}'} \int_{1}^{+\infty} \frac{1}{(z_N^2 + |z'|^2)^{(N + 2s)/2}}dz_N dz' \\
& \leq \Lambda \int_{B_{\eta/\rho}' \setminus B_{\eta/\sqrt{\rho}}'} \frac{1}{|z'|^{N + 2s - 1}} \int_{0}^{+\infty} \frac{1}{(t^2 + 1)^{(N + 2s)/2}}dt dz',
\end{align*}
and from here we conclude that
\begin{equation*}
I_{32} \leq C\Lambda \int_{\eta/\sqrt{\rho}}^{+\infty} r^{-(1 + 2s)} dr \leq C \Lambda \rho^s.
\end{equation*}

Then, collecting the previous estimates, we conclude that
\begin{equation*}\label{I3}
I_3 \leq C \rho^s.
\end{equation*}

For the lower bound, we see that
\begin{equation*}
I_3 \geq - \int_{B_{\eta/\rho}' \setminus B_{\eta/\sqrt{\rho}}'} \int_{-\frac{\eta}{\rho}}^{-1} (1 + z_N)^\tau K^\rho(z)dz,
\end{equation*}
and arguing similarly as before, we arrive at $I_3 \geq -C \rho^s$. Hence, we conclude that $I_3 = O(\rho^s)$. Using this estimate, together with~\eqref{I1} and~\eqref{I2} and replacing them into~\eqref{splittildeI} we conclude that
\begin{equation}\label{tildeI}
\tilde I = \rho^{\tau - 2s} (O(\rho^{1 + \tau}) + O(\rho^{s})),
\end{equation}
where the $O$-term depends on $N, s, \Omega, \eta, \frac{1}{1 + \tau}$.

\medskip
\noindent
\textsl{2.- Estimate for $I_0$.} Rescaling, we have
\begin{align*}
I_0 = & \rho^{\tau - 2s} \mathrm{P.V.} \int_{Q_{\eta/\rho}} [(1 + z_N)_+^\tau - 1] K^\rho(z)dz \\
= & \rho^{\tau - 2s} \mathrm{P.V.} \int_{\R^N} [(1 + z_N)_+^\tau - 1] K^\rho (z)dz 
 + \rho^{\tau - 2s} \int_{Q_{\eta/\rho}^c} [(1 + z_N)_+^\tau - 1] K^\rho(z)dz\\
=: & \rho^{\tau - 2s} c_K(\rho, \tau) + I_{01}.
\end{align*}

Notice that 
$$
-C \Gamma \rho^{\tau - 2s} \rho^{2s} \leq I_{01}.
$$

On the other hand
\begin{align*}
I_{01} \leq & \Gamma \rho^{\tau - 2s} \int_{Q_{\eta/\rho}^c \cap \{ z_N > -1 \} } (1 + z_N)^{\tau} |z|^{-(N + 2s)} dz \\
\leq & C\Gamma \rho^{- 2s} \int_{B_{\eta/\rho}'} \int_{\eta/\rho}^{+\infty} |z|^{-(N + 2s)}dz + C\Gamma \rho^{\tau - 2s} \frac{1}{1 + \tau}\int_{B_{\eta/\rho}'^c} |z'|^{-(N + 2s)}dz' \\
\leq & C \Gamma (1 + \rho^{1 + \tau}).
\end{align*}

Summarizing, we have
\begin{equation*}
I_{01} = O(\rho^\tau).
\end{equation*}

%
Joining this together with~\eqref{tildeI} we conclude the result.
\end{proof}



\section{ Proof of Theorem~\ref{teofbdd}}\label{secteo}

This section is entirely devoted to the proof of Theorem~\ref{teofbdd}. We recall that for a given family of kernels $\mathcal K$ satisfying~\eqref{ellipticity}, we consider the extremal operators associated to this family as 
$$
\mathcal M^{+} u(x) = \sup_{K \in \mathcal K} L_K u(x), \quad \mathcal M^{-} u(x) = \inf_{K \in \mathcal K} L_K u(x).
$$

For every operator with the form~\eqref{opp}, each admissible function $u$ and $x \in \R^N$, we have 
\begin{equation}\label{extre}
 \M^{-} u(x) \leq\I(v+u)(x)-\I v(x) \leq \M^{+} u(x), 
\end{equation}
see~\cite{CS}.

\begin{proof}[Proof of Theorem~\ref{teofbdd}] We prove each case separately. We give the general remark that both sub and supersolutions we construct here are strict.

\medskip
\noindent
\textsl{Case 1:} For $t > 0$ and $s-1 < \gamma <  2s - 1<s$ to be fixed, we consider
\[
U^-_t=td^{s-1} - C_1d^{\gamma},
\]
for some $C_1>0$ to be chosen.

Then, using Proposition~\ref{lema.cota1} together with assumption~\eqref{Kestable} , we have  $c(s-1) = 0$ and $c^+(\gamma)<0$.
Then using \eqref{extre} we find

\begin{align*}
-\I U^-_t +|DU^-_t|^p \leq & -\I(td^{s-1})+C_1M^+(d^\gamma)+|DU^-_t|^p \\
\leq & tO(d^{-1})+C_1c^+(\gamma) d^{\gamma-2s}+|(t (s-1) d^{s - 2} - C_1 \gamma d^{\gamma - 1}) Dd|^p \\
= & tO(d^{-1})+C_1c^+(\gamma) d^{\gamma-2s} \\
& +|t (s-1)|^p d^{(s - 2)p}|1 - C_1 t^{-1} (s-1)^{-1}\gamma d^{\gamma + 1 - s}|^p
\end{align*}

Since $p < p_2$ we can take $\gamma > s - 1$ such that $(s - 2)p > \gamma-2s$ (notice that if $p < p_0$, then $\gamma$ can be taken positive). 

For such a $\gamma$,  we take 
$$
\bar C_1 = |t (s-1)|^p/|c^+(\gamma)|,
$$
to conclude that there exists $\bar c_1 > 0$ such that, for every $\epsilon > 0$, we take $C_1 = \bar C_1 - \epsilon$ in the expression above to obtain that
\[
-\I U^-_t +|DU^-_t|^p \leq tO(d^{-1}) - \bar c_1 \epsilon d^{\gamma - 2s},
\]
for each $d \leq d_\epsilon$ for some $d$ small enough in terms of $\epsilon, t, C_1, \gamma.$ Since $s > 1/2$, and $\gamma < 2s-1$ we can take $d_\epsilon$ smaller to conclude that
\[
-\I U^-_t +|DU^-_t|^p +  \leq - \frac{\bar c_1 \epsilon}{2} d^{\gamma - 2s} \leq - \| f \|_\infty , \quad \mbox{for} \ d(x) \leq d_\epsilon,
\]
and therefore we have constructed a subsolution near the boundary. A straightforward computation tells us that the function
$$
U^- - C\chi_\Omega,
$$
is a viscosity subsolution to the problem in $\Omega$, when $C = C_\epsilon$ is taken large . Thus, we have constructed the subsolution.

In a similar way, we can construct a supersolution in $\Omega$ with the form
\[
U^+_t=td^{s-1} + (\bar C_1+ \epsilon) d^{\gamma} + \C_\epsilon \chi_\Omega.
\]

We can apply the Perron method from Section \ref{secper} to conclude the existence of a solution to the problem \eqref{eq} satisfying the desired rate near the boundary: if $p < p_0$, then we can take $\gamma > 0$ in the above analysis.

Similar arguments hold for $t<0$.

\bigskip

\noindent
\textit{Case 2:} Let $\beta$ as in the statement of the theorem. Hence, $-1< \beta <s-1$ and for some $s-1 < \gamma < 0$ and $T, C_1 > 0$ to be fixed, denote
\[
U^- = Td^\beta - d^\gamma.
\]

We invoke Proposition~\ref{lema.cota1}  again, noticing that if~\eqref{Kestable} holds, and by choice of  $\beta$ and range of $\gamma$ we have $c(\beta)<0$ and $c^+(\gamma)<0$. Then using again \eqref{extre} we find

\begin{align*}
-\I U^-+|DU^-|^p 
\leq & -c(\beta) T d^{\beta - 2s} + O(d^{\beta - s}) + c^+(\gamma) d^{\gamma - 2s} \\
& + (T |\beta|)^p d^{(\beta - 1)p}|1  -  T^{-1} \beta^{-1} \gamma d^{\gamma - \beta}|^p \\
\leq & (-c(\beta) T  + (T |\beta|)^p) d^{\beta - 2s} + O(d^{\beta - s}) +  c^+(\gamma) d^{\gamma - 2s}.
\end{align*}

Then, taking $T = \bar T$ such that
$$
c(\beta) \bar T = \bar T^p |\beta|^p,
$$
 conclude that
\begin{equation*}
-\I U^- + |DU^-|^p \leq O(d^{\beta - s})+ c^+(\gamma) d^{\gamma - 2s},
\end{equation*}
in a neighborhood of $\partial \Omega$. Note that it is possible to fix $\gamma$ close enough to $s - 1$ in order to have $\gamma - 2s < \beta - s$, and then, for each point close to the boundary we have
\begin{align*}
-\I U^- + |DU^-|^p \leq  \frac{c^+(\gamma)}{2} d^{\gamma - 2s}.
\end{align*}

We conclude the existence of a subsolution in $\Omega$ in the same way as before.

For the supersolution, we consider
\begin{equation*}
U^+ = \bar T d^{\beta} + d^\gamma,
\end{equation*}
and proceed as before.

\noindent
\textit{Case 3:} The proof is similar to the previous case, but we provide the details for completeness. Since $p>p_2$ we have $s-1<\beta<0$.
For $T, C>0$ and $\gamma \in (0,2s-1)$ to be fixed, notice that $c(\gamma)<0$, since $2s-1<s$.
Now we define 
\[
U =-Td^{\beta} - Cd^\gamma.
\]

Then, writing as before $\tilde \I$ as the operator $-  I (- \cdot)$, and using Proposition~\ref{lema.cota1} and~\eqref{Kestable}, we have again using \eqref{extre}
\begin{align*}
-\I U  + |DU|^p & \leq T \tilde \I (d^\beta ) + C \M^+ (d^\gamma) + |T \beta d^{\beta - 1} + C\gamma d^{\gamma - 1}|^p \\
& \leq d^{\beta - 2s} (T \tilde c(\beta) + O(d^s)) + C c^+(\gamma) d^{\gamma - 2s} + |T \beta d^{\beta - 1} + C\gamma d^{\gamma - 1}|^p,
\end{align*}
where $\tilde c(\beta) < 0$  since $s-1<\beta<0$. At this point, we notice that since $\beta < 0$ and $C, \gamma > 0$, we have 
$$
|T \beta d^{\beta - 1} + C\gamma d^{\gamma - 1}| \leq T |\beta| d^{\beta - 1},
$$
for all $d = d(x)$ small enough. Using this, we get that
\begin{align*}
-\I U  + |DU|^p \leq d^{\beta - 2s} (T \tilde c(\beta) + O(d^s)) + C c^+(\gamma) d^{\gamma - 2s} + T^p |\beta|^p d^{(\beta - 1)p},
\end{align*}
and fixing $T = T^* := (-\tilde c(\beta) |\beta|^p)^{\frac{1}{p - 1}} > 0$, since $\beta - 2s = (\beta - 1)p$ we conclude that
\begin{align*}
-\I U  + |DU|^p \leq O(d^{\beta - s}) + Cc^+(\gamma) d^{\gamma - 2s} ,
\end{align*}
Since we also have $\beta - s>-1>\gamma-2s$.  fixing $C > 0$ large enough. Then, by similar arguments used before, we   can fix$C > 0$ large enough to find  that $U$ is a subsolution.

Now, for $C > 0$ and $\gamma > 0$ to be fixed, consider the function $V = -T^* d^\beta+ C d^\gamma$. Notice that $U \leq V$.  By a similar computation as above
we get that $V$ is a supersolution for the problem.

\end{proof}

\begin{remark}\label{rmk1}
Some remarks concerning the asymptotics as $s \to 1$ for $p \in (1,2)$ fixed. By the construction of the barriers in Theorem~\ref{teofbdd}, we see that the one-parameter solutions, and the negative scale solutions (Cases 1 and 3) shall always converge to the unique, bounded solution to the problem~\eqref{eqLLloc} with $u = 0$ on $\partial \Omega$. On the other hand, positive scale solutions (Case 2) converge to the unique large solution to~\eqref{eqLLloc}. This is a consequence of well-known stability results of viscosity solutions, the estimates for each of the solutions in Cases 1, 2 and 3, and the uniqueness of the limit equation.

Concerning ``critical" cases, we do not know if there exists blow-up solutions for the case $p=p_2$ (and $s$ fixed). In fact, the barriers constructed provide estimates that imply that the solutions found for $p < p_2$ and $p > p_2$ tend to the bounded solution of the limit problem as $p \to p_2^\pm$. For the case $p = 2s$ it was neither possible to construct large solutions by approximation as $p \to 2s^-$ ($s$ fixed). The local setting suggest to search for large solutions with logarithmic profile, but we leave that analysis for a future work. 
%
%
%
\end{remark}

\section{Extensions}\label{secext}
In this section we provide a discussion about possible extensions of Theorem~\ref{teofbdd} for more general operators and data.

First recall that the estimates found in Theorem~\ref{teofbdd} are crucial in order to prove Theorem \ref{teofbdd}. In order to obtain a similar result than Theorem~\ref{teofbdd} for general kernels, it suffices that the family of kernels $K$ satisfies the requirement
\begin{equation}\label{casic}
c(d(x), \tau) = c(\tau) + O(d(x)^\alpha),
\end{equation}
for some $\alpha > 0$ and some $c: (-1, 2s) \to \R$ that is independent of $x$.

For example, this condition holds when $K = K_a$ has the form
\begin{equation}\label{Ka}
K_a(z) = \frac{a(z)}{|z|^{N + 2s}}, \quad z \neq 0,
\end{equation}
where $a: \R^N \to \R$ is a nonnegative, measurable function with a uniform modulus of continuity at $z = 0$. More specifically, if we consider the class $a \in \mathcal A$ such that there exists $C, r > 0$ and $\alpha' > 0$ such that
\begin{equation}\label{KaHolder}
\sup_{a \in \mathcal A } \{ |a(z) - a(0) | \} \leq C|z|^{\alpha'}, \quad \mbox{for all} \ z \in B_r, 
\end{equation}
and use the class of kernels $\{ K_a \}_{a \in \mathcal A}$,
then expansion~\eqref{casic} holds for some $0 < \alpha < \alpha'$. 

In fact, the key step to arrive at~\eqref{casic} comes by the expansion, for $K = K_a$, given by
\begin{align*}
c_K(\rho, \tau) = & -C_{N, s}^{-1} a(0) (-\Delta)^s (x_N)_+^\tau (e_N) \\
& +  \mathrm{P.V.} \int_{\R^N} [(1 + z_N)_+^\tau - 1] \frac{(a(\rho z) - a(0))dz}{|z|^{N + 2s}}.
\end{align*}

The first term in the right-hand side is independent of $\rho$. The regularity assumption on $a$ allows us to control the second term in the righ-hand side as an error term of order $O(\rho^{\alpha})$ for some $\alpha < \alpha'$. For this, we divide the integrand as $\R^N = B_{R_\rho} \cup B_{R_\rho}^c$ for some $R_\rho$ large depending on $\rho$. The exterior part is controlled by the tails of the kernel and the boundedness of $a$ (ellipticity). The inner part requires the modulus of continuity $a$, for which we have the restriction that $R_\rho << \rho^{-1}$. See Lemma 3.1 in~\cite{DQT} for details.

{Concerning the exterior data, we reduce the problem to the homogeneous case by inserting it in the source term. namely, for $\varphi \in L^1_\omega(\Omega^c)$, denote
\begin{align*}
\tilde\varphi(x)=\left\{\begin{array}{ll}
0 &\text{for }x\in\Omega,\\
\varphi(x)&\text{for }x\in\Omega^c.
\end{array}\right.
\end{align*}

Observe that $\M^\pm \varphi(x)$ is well defined for $x \in \Omega$. Furthermore note that if $\varphi$ is smooth and extends smoothly to 0 in $\bar\Omega$, then $\M^\pm \varphi $ is also smooth and bounded in $\Omega$.

Now, if $u$ is a solution of
\begin{equation}\label{eq2}
\left \{ \begin{array}{rl} - \I (u) + |Du|^p + \lambda u = f \quad & \mbox{in} \ \Omega, \\
u  = \varphi \quad & \mbox{in} \ \Omega^c,  
\end{array} \right .
\end{equation}
then $\tilde u=u-\tilde\varphi$ is a solution of 
\begin{align}\label{eqgeneral0}
\left\{\begin{array}{rl}
\tilde{\mathcal I}(\tilde u)+|D \tilde u|^p+\lambda u=f,& \text{in }\Omega,\\
\tilde u=0,& \text{in }\Omega^c
\end{array}\right.
\end{align}
where $\tilde{\mathcal I}(v)=\mathcal I(v+\tilde \varphi)$ and satisfies
\[
\I v+\M^-\tilde\varphi \leq \tilde{\mathcal I}(v)\leq \I v+\M^+\tilde\varphi.
\]

%

These relations allow us to always consider 0 as the boundary data and impose the extra assumptions on the new right-hand side involving $f$ and the extremal operators evaluated at $\tilde \varphi$. Since our method is based on sub and supersolutions, we only need to obtain suitable inequalities, the ones imposing some restriction to use the method.  For example, as mentioned earlier, if $\tilde\varphi$ is smooth in an exterior neighborhood of $\partial\Omega$ then $\tilde f$ will be bounded as long as $f$ is bounded and hence all the results of Theorem \ref{teofbdd} applies. On the other hand, if we allow the right hand side to be unbounded, then that would require less regularity of the extension $\tilde\varphi$, for example if $\tilde\varphi$ is just bounded then $\M^\pm(\varphi)$ would be of order $d^{-2s}$. 

For our more general theorem below we introduce de following assumption 

 $(H1)$ $\quad  \mbox{there exists}  \quad \eta>0\quad  \mbox{such that}$
$$\quad\limsup\limits_{x\to \partial \Omega} d(x)^{s+1-\eta} f(x)< \infty ,\quad \mbox{ and }\quad \liminf\limits_{x\to \partial \Omega} d(x)^{s+1-\eta} f(x)>- \infty  ,$$
  
or
  
$(H2)$ $\quad  \mbox{there exists}  \quad \eta>0\quad  \mbox{such that}$
$$\quad\limsup\limits_{x\to \partial \Omega} d(x)^{2s-\eta} f(x)< \infty ,\quad \mbox{ and }\quad \liminf\limits_{x\to \partial \Omega} d(x)^{2s-\eta} f(x)>- \infty  ,$$

%
%

\begin{teo}\label{Teoprin}  Let $s \in (1/2,1)$, $0 < p < 2s$, $\Omega \subset \R^N$ be a bounded domain with $C^2$ boundary, $f\in C(\Omega)$.

Assume $\mathcal K$ as in Theorem~\ref{teofbdd}, or with the form~\eqref{Ka}-\eqref{KaHolder}. Let $\I$ be a nonlinear operator with the form~\eqref{opp}, and let $p_i$ be defined as in~\eqref{ppp}, $i=0,1,2$. Assume that $\lambda>-\lambda_0(\I)$

Then, we have the following existence results:

\medskip
\noindent
\textsl{1.- One parameter family of solutions (close to $s$-harmonic):} If $0 < p < p_2$ and $f$ satisfies $(H1)$ , there exists $\sigma > 0$ and a family of solutions $\{ u_t \}_{t \in \R, t \neq 0} \subset C^{\sigma}(\Omega)$ to \eqref{eq2}, such that for each $t$ we have
$$
d^{1 - s} u_t(x) - t = O(d^{\gamma}),  
$$
for some $\gamma > 0$ depending on $p$. In particular, if $t_1 < t_2$, then
$$
u_{t_1} < u_{t_2} \quad \mbox{in} \ \Omega.
$$

\medskip
\noindent
\textsl{2.- Positive scale solution:} If $p_1 < p < p_2$ and $f$ satisfies $(H1)$, then there exists $\sigma > 0$ and a constant $T > 0$ and a function $u \in C^{\sigma}(\Omega)$ solving~\eqref{eq2} such that
$$
d(x)^{-\beta}u(x) - T = O(d(x)^\gamma),
$$ 
for some $\gamma > 0$.
\medskip

\noindent 
\textsl{3.- Negative scale solution:} For $p_2 < p < 2s$ and $f$ satisfies $(H2)$, then there exist $\sigma > 0$, $T > 0$ and a solution $u \in C^{\sigma}(\Omega)$ of \eqref{eq2} such that
$$
d^{-\beta}(x)u(x) + T = O(d(x)^\gamma), 
$$
for some $\gamma > 0$.

\end{teo}

\begin{remark}   Theorem  \ref{Teoprin}  can be used for equation like \eqref{eq2}  by the discussion before the Theorem  \ref{Teoprin} and replacing $f$ by $\tilde f$ .
In particular,  if $\tilde\varphi $  is just bounded then $\tilde f$ is of order $d^{-2s}$ as mentioned above and therefore satisfies  $(H1)$.  
\end{remark}

The proof of Theorem \ref{Teoprin} follows exactly as Theorem \ref{teofbdd} the point here is that we can choose $\gamma$ such that 
$(\gamma-2s)\leq -(s+1)+\eta$ in the case 1) and 2) and  such that $\gamma\leq \eta$  and then we can take $\gamma$ $(\gamma-2s)\leq -2s+\eta$  in case 3). 
Then using  (H1) (or (H2) for case 3)) we get our barrier in each of cases as in the proof of Theorem \ref{teofbdd} .

\bigskip

\noindent {\bf Acknowledgements.} 

G. D\'avila was partially supported by Fondecyt Grant 1190209

A. Q. was partially supported by Fondecyt Grant No. 1190282 and Programa Basal, CMM. U. de Chile.

E. T. was partially supported by Fondecyt Grant No. 1201897


\end{document}